\documentclass{proc-l}

\usepackage{enumerate}
\usepackage{mdwlist}
\usepackage{cleveref}

\newcommand{\BBN}{\mathbb N}

\newcommand{\BBR}{\mathbb R}

\newcommand{\expon}{Q}
\newcommand{\map}{f}
\newcommand{\dfn}{h}

\newcommand{\meas}{\mu}																							
\newcommand{\mapmeas}{\nu}
\newcommand{\genpullmeas}[2]{{#1}^*#2}															
\newcommand{\pullmeas}{\genpullmeas{\map}{\mapmeas}}

\newcommand{\lloc}[1]{ L^{#1}_{\operatorname{loc}}}									
\newcommand{\norm}[1]{||#1||_{\expon}}

\newcommand{\esssup}{\operatorname{ess }\sup}

\newcommand{\lebesgue}[1]{\operatorname{m}_{#1}}										
\newcommand{\dist}{\operatorname{dist}}															
\newcommand{\diam}{\operatorname{diam}}		

\newcommand{\lip}[1]{\operatorname{Lip} {#1}}
\newcommand{\aplip}[1]{\operatorname{apLip} {#1}}
\newcommand{\abscurves}{\operatorname{\mathcal C_{ABS}}}	         
\newcommand{\rectcurves}{\operatorname{\mathcal C_{RECT}}}				
\newcommand{\allcurves}{\operatorname{\mathcal C}}
\newcommand{\bigcurves}[1]{\operatorname{\mathcal C}_{#1}}
\newcommand{\acon}[3]{\operatorname{\mathcal A}(#1,#2,#3)}															
\newcommand{\aacon}{\mathcal A}
\newcommand{\charfcn}[1]{\mathbf{1}_{#1}}														
\newcommand{\restr}[2][\big]{\kern -.1em #1|_{#2}}									
\newcommand{\mrestrict}[1]{\lfloor_{#1}}           									
\newcommand{\restrict}[2][\big]{\kern -.1em #1|_{#2}}

\renewcommand{\mod}[1][\expon]{\operatorname{Mod}_{#1}}                    

\newcommand{\grad}[1]{g_{#1}}

\newcommand{\jac}[1][\map]{J_{#1}}                                        

\newcommand{\arc}[1]{{#1}^{s}}

\newcommand{\cheegersob}[2]{H^{#1,#2}}

\newcommand{\newtonsob}[2]{N^{#1,#2}}
\newcommand{\newtonsobloc}[2]{N_{\operatorname{loc}}^{#1,#2}}

\theoremstyle{definition}

\newtheorem{definition}{Definition}[section]

\theoremstyle{plain}

\newtheorem{lemma}[definition]{Lemma}
\newtheorem{theorem}[definition]{Theorem}
\newtheorem{corollary}[definition]{Corollary}
\newtheorem{proposition}[definition]{Proposition}

\theoremstyle{remark}

\newtheorem{remark}[definition]{Remark}

\crefname{chapter}{Chapter}{Chapters}
\crefname{section}{Section}{Sections}
\crefname{subsection}{Section}{Sections}
\crefname{subsubsection}{Section}{Sections}

\crefname{theorem}{Theorem}{Theorems}
\crefname{lemma}{Lemma}{Lemmas}
\crefname{proposition}{Proposition}{Propositions}
\crefname{corollary}{Corollary}{Corollaries}
\crefname{definition}{Definition}{Definitions}
\crefname{remark}{Remark}{Remarks}
\crefname{example}{Example}{Examples}
\crefname{observation}{Observation}{Observations}

\begin{document}

\title[Geometric and analytic quasiconformality]{Geometric and analytic quasiconformality in metric measure spaces}
\author{Marshall Williams}
\address{Department of Mathematics, Statistics, and Computer Science \\University of Illinois at Chicago\\}
\thanks{Partially supported under NSF awards \#0602191, \#0353549 and \#0349290.}


\begin{abstract}
We prove the equivalence between geometric and analytic definitions of quasiconformality for a homeomorphism $\map\colon X\rightarrow Y$ between arbitrary locally finite separable metric measure spaces, assuming no metric hypotheses on either space.  When $X$ and $Y$ have locally $Q$-bounded geometry and $Y$ is contained in an Alexandrov space of curvature bounded above, the sharpness of our results implies that, as in the classical case, the modular and pointwise outer dilatations of $\map$ are related by $K_O(\map)= \esssup  H_O(x,\map)$.   
\end{abstract}

\maketitle

\section{Introduction}
\label{introduction}
In the last few decades, there has been an increasing interest in 
the extension of the theory of quasiconformal mappings to metric measure spaces.  Let $\map\colon X\rightarrow Y$ be a homeomorphism between metric measure spaces $(X,\meas)$ and $(Y,\mapmeas)$.  
Heinonen, Koskela, Shanmugalingam, and Tyson \cite[Theorem 9.8]{HKST} proved that if $X$ and $Y$ have locally $Q$-bounded geometry (i.e., loosely speaking, they are uniformly locally $Q$-regular and $Q$-Loewner), then the usual definitions (metric, geometric, and analytic) of quasiconformality are equivalent, quantitatively, to each other and to local quasisymmetry.  In this generality, the analytic definition is formulated via the Newton-Sobolev classes $\newtonsobloc{1}{Q}(X,Y)$ introduced in \cite{HKST}.

Without the Poincar\'e inequality, the equivalence of the definitions breaks down. For example, if $X$ and $Y$ have no rectifiable curves, then the geometric definition becomes vacuous.  Still, some relationships do persist between the different notions.  Tyson \cite{Tyson} proved that if $X$ and $Y$ are $Q$-regular, then quasisymmetry implies the geometric definition.  Newton-Sobolev regularity was proved for quasisymmetric mappings in \cite[Theorem 8.8]{HKST}, and generalized to metrically quasiconformal mappings by Balogh, Koskela, and Rogovin \cite{BaloghKoskelaRogovin}.  The latter result yields  the lower half of the geometric definition as well --- the ``$K_O$-inequality''\cite[Remark 4.3]{BaloghKoskelaRogovin}.

The purpose of this paper is to generalize and sharpen the equivalence between the geometric and analytic definitions of quasiconformality.  

\subsection*{General equivalence of the definitions.}
Our main result is that the analytic definition, formulated in terms of the minimal upper gradient $\grad{\map}$ and the volume derivative $\jac$, is precisely equivalent to the $K_O$-inequality, in very great generality.
\begin{theorem}
\label{quasiconformal}
Let $Q>1$, let $X$ and $Y$ be separable, locally finite metric measure spaces, and let $\map\colon X\rightarrow Y$ be a homeomorphism.  Then the following two conditions are equivalent, with the same constant $K$.
\begin{enumerate}[(I)]
\item \label{adef}
$\map\in\newtonsobloc{1}{Q}(X,Y)$, and for $\meas$-almost every $x\in X$, 
\begin{equation*}
\grad{\map}(x)^Q\leq K \jac(x)\text{.}
\end{equation*}
\item \label{gdef} 
For every family $\Gamma$ of curves in $X$,
\begin{equation*}
\mod(\Gamma)\leq K \mod(\map(\Gamma))\text{.}
\end{equation*}
\newcounter{tempenumi}
\setcounter{tempenumi}{\value{enumi}}
\end{enumerate}
\end{theorem}

The fact that condition \eqref{adef} implies \eqref{gdef} is somewhat standard, and has been alluded to by others \cite[Remark 4.3]{BaloghKoskelaRogovin}, though we prove it for completeness, as we know of no proof in the literature for metric spaces.  The substantive part of \cref{quasiconformal} is the converse implication.  This has only been proved sharply in the classical case;  the analogous result for $Q$-bounded geometry, \cite[Theorem 9.8]{HKST}, is quantitative, and is proved indirectly via the metric definition and quasisymmetry.  Our proof is based on a characterization of the $L^p$ norm of $\grad{\map}$ via the modulus of certain curve families, and requires no metric assumptions on either $X$ or $Y$.  The argument appears to be new even for $X=Y=\BBR^n$.

A key difference from the Euclidean, and even Loewner, setting is that in our generality, \cref{quasiconformal} is fundamentally one-sided; the equivalent conditions in the theorem typically do not imply the reverse ``$K_I$-inequality'', even when $X$ and $Y$ are $Q$-regular.  We discuss counterexamples in \cref{onesided} below.  

\subsection*{Annular quasiconformality and Tyson's Theorem.}
Under only the additional assumption of a doubling condition on the measure $\mapmeas$, 
we  show that infinitesimal control of the modulus of certain annular condensers implies  conditions \eqref{adef} and \eqref{gdef}.  Recall that $\mapmeas$ is doubling if there is a constant $C>0$ such that
\begin{equation*}
\mapmeas(B(y,2r))\leq C\mapmeas(B(y,r))
\end{equation*}
is satisfied for every $y\in Y$ and $r>0$.
For any metric space $Z$, any $z\in Z$, and any $s>r>0$, we define the \textit{annular condenser} $\acon{z}{r}{s}$ to be the family of curves intersecting both $B_r(z)$ and $Z\backslash B_s(z)$.
\begin{theorem}
\label{annularquasiconformal}
Let $Q$, $X$, $Y$ and $\map$ be as in \cref{quasiconformal}, and suppose that the measure $\mapmeas$ is doubling.   Then  conditions \eqref{adef} and \eqref{gdef} 
are quantitatively equivalent to
\begin{enumerate}[(I)]
\setcounter{enumi}{\value{tempenumi}}
\item \label{cdef}
There is some $\lambda>1$ and $K'\geq 1$ such that for every $y\in Y$,
\begin{equation*}
\liminf_{r\rightarrow 0} \frac{r^Q\mod(\map^{-1}(\acon{y}{r}{\lambda r}))}{\mapmeas(B_r(y))}\leq K'\text{.}
\end{equation*}\setcounter{tempenumi}{\value{enumi}}
\end{enumerate}
\end{theorem}

As we discuss in \cref{applicationsremark}, there are a number of hypotheses on $\map$ that guarantee condition \eqref{cdef}. These include ``annular'' and ``ring'' definitions of quasiconformality. Quasisymmetry also implies condition \eqref{cdef}, so that \cref{annularquasiconformal} gives a short (though not entirely new in concept) proof for Tyson's theorem on the geometric quasiconformality of quasisymmetric mappings.

\subsection*{Locally $Q$-bounded geometry and sharp equivalence of dilatations.}
In general, the minimal weak upper gradient $\grad{\map}$ need not carry much geometric information; if $X$ has no rectifiable curves, for example, then $\grad{\map}=0$.  If $X$ has locally $Q$-bounded geometry, however, it follows from results of Cheeger \cite{Cheeger} that $\grad{\map}$ is comparable to the pointwise Lipschitz constant
\begin{equation*}
\lip{\map}(x) = \limsup_{x'\rightarrow x, x'\neq x} \frac{ |\map(x')- \map(x)|}{|x'-x|} \text{.}
\end{equation*}
As a result, the inequality in \eqref{adef} is quantitatively equivalent to the inequality
\begin{equation*}
\lip{\map}(x)^Q\leq K'' \jac(x)\text{,}
\end{equation*}
which is in fact how the analytic definition appears in \cite[Theorem 9.8]{HKST}.  Moreover, if the target $Y$ has curvature bounded above in the sense of Alexandrov, then results of Cheeger \cite{Cheeger}, Keith \cite{KeithPI}, and Ohta \cite{Ohta} show (\cref{piqslip} below) that $\lip{\map}=\grad{\map}$.  

The modular outer dilatation $K_O(\map)$ of $\map$ is the infimal value of $K$ satisfying condition \eqref{gdef}.  For each $x\in X$, define the pointwise outer dilatation $H_O(x,\map)$ of $\map$ at $x$ to be
\[
H_O(x,\map)=\limsup_{x'\rightarrow x, x'\neq x} \frac{|\map(x')-\map(x)|^Q\meas(B_{|x'-x|}(x))}{|x'-x|^Q\mapmeas(\map(B_{|x'-x|}(x)))}\text{.}
\]
The latter definition is motivated by the outer dilatation $H_O(\map'(x))$ of the derivative in the classical case (cf.\ \cite[Definition 14.1]{Vaisala}). At almost every $x\in X$, $H_O(x,\map)=\lip{\map}(x)^Q/\jac(x)$ (see the proof of \cref{sharpening} in \cref{lipminimal}).  The equivalence of the definitions in \cite[Theorem 9.8]{HKST} therefore indicates a quantitative relationship between $K_O(\map)$ and $\esssup H_O(x, \map)$.  Via \cref{quasiconformal} and the comparability of $\grad{\map}$ and $\lip{\map}$, we obtain a more precise statement relating  the two dilatations.
\begin{theorem}
\label{sharpening}
Let $X$ and $Y$ be metric measure spaces of locally $Q$-bounded geometry, for $Q>1$, and let $\map\colon X\rightarrow Y$ satisfy the conditions in \cref{quasiconformal}.  Then 
\[
\frac{1}{C}\esssup_{x\in X} H_O(x,\map)\leq K_O(\map)\leq \esssup_{x\in X} H_O(x,\map)\text{,}
\]
where $C$ is a constant depending only on the data of $X$.  If, in addition, $Y$ is isometrically contained in a locally compact, locally geodesically complete metric space of curvature bounded above, then $K_O(\map)=\esssup H_O(x,\map)$.
\end{theorem}
The last statement generalizes the classical equivalence between the outer dilatations \cite[Theorem 34.4]{Vaisala}.

\subsection*{Organization of the paper.}
In \cref{preliminaries} we establish notation and recall preliminary definitions and properties of curves, modulus, and upper gradients.  \cref{uganalysis} is devoted to proving a number of facts about weak upper gradients, leading up to our main technical tool, \cref{modgradthm}, which expresses the $L^p$ norm of a minimal upper gradient via curve modulus.  In \cref{quasiconformality} we prove  \cref{quasiconformal,annularquasiconformal}. In \cref{lipminimal} we prove \cref{sharpening}.

\subsection*{Acknowledgments}
I thank my advisor, Mario Bonk, as well as Pekka Pankka and Stefan Wenger, for reviewing early versions of some of these results, and for many helpful discussions.  I am also indebted greatly to my late advisor, Juha Heinonen, for much guidance and encouragement in entering this field. 
\section{Preliminaries and Notation.}
\label{preliminaries}
\renewcommand{\expon}{p}
Throughout this paper, $X = (X,\dist_X,\meas)$ and $Y = (Y,\dist_Y,\mapmeas)$ are separable metric measure spaces, and $Z=(Z,\dist_Z)$ is an arbitrary metric space.  Here $\meas$ and $\mapmeas$ are assumed to be locally finite Borel regular outer measures on $X$ and $Y$ that are positive on open sets.
We  write $|x_1-x_2|=\dist(x_1,x_2)$ when the metric is clear from context. We also sometimes write $d_{x_1}(x_2)=|x_1-x_2|$.  We denote by $B_r(x)$ 
the closed ball of radius $r$ centered at $x$.  

The characteristic function of a subset $A\subseteq X$ is denoted $\charfcn{A}$.

Unless otherwise specified, $\map\colon X\rightarrow Y$ is a homeomorphism, and $\dfn\colon X\rightarrow Z$ is a Borel map.  If $\meas$ is finite, we say $\dfn\in L^p(X,Z)$ if $d_z\circ \dfn\in L^p(X)$ for some (and therefore every) $z\in Z$, and we define $\lloc{p}(X,Z)$ similarly.

By $\pullmeas$ we denote the pushforward of $\mapmeas$ by $\map^{-1}$; that is, $\pullmeas(A) = \mapmeas(\map(A))$ for every $A \subseteq X$.  The restriction of $\meas$ to $A$ is denoted $\meas\mrestrict{A}$. The Radon-Nikodym derivative of $\pullmeas$ with respect to $\meas$ is $\jac$.

A \textit{curve} is a continuous map $\gamma\colon [a,b]\rightarrow X$, where $[a,b]\subset \BBR$ is a closed interval.  When there is no chance for confusion, we always use $[a,b]$ to denote the parametrizing interval of $\gamma$.  A \textit{subcurve} of $\gamma$ is the restriction $\restrict[\gamma]{[c,d]}$ of $\gamma$ to a closed subinterval $[c,d]\subseteq [a,b]$. We denote by $\allcurves(X)$ the set of curves in $X$.

\subsection*{Arc length}

Following \cite{Duda}, we define the variation function $v_\gamma\colon [a,b] \rightarrow [0,\infty]$ by 
\begin{equation*}
v_\gamma(t)=\sup_{a\leq a_1\leq b_1\leq\dotsb\leq a_n\leq b_n\leq t} \sum_{i=1}^n |\gamma(b_i)-\gamma(a_i)|\text{.}
\end{equation*}

The \textit{length} $l(\gamma)$ of $\gamma$ is  $l(\gamma)=v_\gamma(b)$.  If $\gamma$ has finite length, we say that $\gamma$ is \textit{rectifiable}, and denote the set of rectifiable curves in $X$ by $\rectcurves(X)$.  The arc-length parametrization of such a curve is $\arc{\gamma}\colon [0,l(\gamma)]\rightarrow X$, and is defined uniquely by the equation $\arc{\gamma}\circ v_\gamma = \gamma$. The integral of a Borel function $\rho$ along  $\gamma$ is 
\[ \int_\gamma \rho \,ds = \int_0^{l(\gamma)} \rho(\arc{\gamma}(t))\,dt\text{.}\]

A curve $\gamma$ is \textit{absolutely continuous} if 
$v_\gamma$ is absolutely continuous. Via the chain rule, we then have
\begin{equation}
\label{nonparam}
\int_\gamma \rho \, ds
=\int_a^b \rho(\gamma(t)) v_\gamma'(t)\,dt\text{.}
\end{equation}
We denote the family of absolutely continuous curves by $\abscurves(X)\subset\rectcurves(X)$.  Note that $\arc{\gamma}\in\abscurves(X)$, always.

For $\gamma\in \abscurves(X)$, the metric derivative studied in \cite{AmbrosioKirchheimRect} and \cite{Kirchheim} coincides with $v_\gamma'$ \cite[Remark 3.4]{Duda}, which immediately implies the following useful fact.

\begin{lemma}
\label{curverestrict}
Let  $\gamma_1,\gamma_2\in\abscurves(X)$, with each curve parametrized by $[a,b]$, and suppose $\restrict[\gamma_1]{A}=\restrict[\gamma_2]{A}$ for a measurable subset $A\subset [a,b]$.  Then $v_{\gamma_1}'(t)=v_{\gamma_2}'(t)$ for almost every $t\in A$.  
\end{lemma}

We say a function $\dfn$ is continuous along the curve $\gamma$ if $\dfn(\gamma)= \dfn\circ\gamma$ is continuous. 
When $\gamma$ is rectifiable, we say $\dfn$ is \textit{absolutely continuous on $\gamma$} if $\dfn(\arc{\gamma})\in\abscurves(Z)$.  Note that if $\dfn$ is absolutely continuous on $\gamma$, and $\gamma\in\abscurves(X)$, then $\dfn(\gamma)\in \abscurves(Z)$.

A detailed discussion of arc-length can be found in \cite[Chapter 1]{Vaisala}.  For much more on absolutely continuous maps into metric spaces, see \cite{Duda}.

\subsection*{Curve modulus.}
\renewcommand{\expon}{p}

Let $\Gamma$ a family of curves in $X$.  A Borel function $\rho\colon X\rightarrow [0,\infty]$ is said to be \textit{admissible} for $\Gamma$ if for every rectifiable $\gamma\in \Gamma$,
\begin{equation}
\label{admissibility}
\int_\gamma \rho\,ds\geq 1\text{.}
\end{equation}
The \textit{$p$-modulus} of $\Gamma$ is 
\begin{equation*}
\mod(\Gamma) = \inf \left\{ \int_X \rho^p\,d\meas:\text{$\rho$ is admissible for $\Gamma$.} \right\} \text{.}
\end{equation*}

A property holds for \textit{$p$-almost every} curve, or simply almost every curve if $p$ is understood, if the property fails only on a family $\Gamma$ such that $\mod(\Gamma)=0$.  If $A\subset X$,  and $\gamma$ is rectifiable, $\gamma$ \textit{has positive length} in $A$ if $\lebesgue{1}({\arc{\gamma}}^{-1}(A))>0$, and \textit{has length $0$} in $A$ otherwise. A curve family $\Gamma$  is \textit{minorized} by $\tilde{\Gamma}$ if every curve in $\Gamma$ has a subcurve in $\tilde{\Gamma}$.

The following results are standard properties of $\mod$, which can be found, for example, in \cite[Chapter 1]{Fuglede}.
\begin{lemma}
\label{properties}
The $p$-modulus has the following properties:
\begin{enumerate}[(i)]
\item The function $\mod\colon \allcurves(X)\rightarrow [0,\infty]$ is an outer measure.
\item \label{minorize} If $\Gamma$ is minorized by $\tilde{\Gamma}$, then $\mod(\tilde{\Gamma})\geq \mod(\Gamma)$.
\item \label{intaecurve} Let $\rho\in\lloc{p}(X)$.  Then $\rho$ is integrable along almost every curve in $\allcurves(X)$. 
\item \label{fuconverge} Let $\{\rho_i\}$ be a sequence of Borel functions in $\lloc{p}(X)$, converging locally in $\lloc{p}(X)$ to $\rho$.  Then there is a subsequence $\{\rho_{i_k}\}$ such that on almost every curve $\gamma\in\rectcurves(X)$, 
\begin{equation*}
\lim_{k\rightarrow\infty}\int_\gamma |\rho_{i_k}-\rho|\,ds = 0\text{.}
\end{equation*}
\item \label{aelength0} Let $E\subset X$ with $\meas(E)=0$.  Then almost every curve has length $0$ in $A$.
\end{enumerate}
\end{lemma}

Note that parts \eqref{intaecurve} and \eqref{fuconverge} in \cref{properties} are stated in \cite{Fuglede} only for the case where the functions $\rho_i$ are actually in $L^p(X)$, not merely in $\lloc{p}(X)$.  The statements immediately generalize, though, via the separability of $X$, local finiteness of $\meas$, and countable subadditivity of $\mod$.
\subsection*{Upper gradients}

A Borel function $g\colon X\rightarrow \BBR$ is called an \textit{upper gradient} for $\dfn$ if for every curve $\gamma\in\rectcurves{X}$, we have the inequality
\begin{equation}
\label{ugdefeq}
\int_\gamma g\,ds\geq|\dfn(\gamma(b))-\dfn(\gamma(a))|\text{.}
\end{equation}
If inequality \eqref{ugdefeq} merely holds for $p$-almost every curve, then $g$ is called a \textit{$p$-weak upper gradient} for $\dfn$.  When the exponent $p$ is clear, we omit it.

By \cite[Lemma 2.4]{KoskelaMacManus}, $\dfn$ has a weak upper gradient in $\lloc{p}(X)$ if and only if it has an actual upper gradient in $\lloc{p}(X)$.

A weak upper gradient $g$ of $\dfn$ is \textit{minimal} if  for every weak upper gradient $\tilde{g}$ of $\dfn$, $\tilde{g}\geq g$ $\meas$-almost everywhere.  If $\dfn$ has an upper gradient in $\lloc{p}(X)$, then $\dfn$ has a unique (up to sets of $\mu$-measure $0$) minimal $p$-weak upper gradient \cite[Theorem 7.16]{Hajlaszsurvey}.  In this situation, we denote the minimal upper gradient by $\grad{\dfn}$.

A function $\dfn\in L^p(X, Z)$ with an upper gradient in $L^p(X)$ is said to be in the Newton-Sobolev class $\newtonsob{1}{p}(X,Z)$, and we define $\newtonsobloc{1}{p}(X,Z)$ similarly.

\begin{remark}
\label{paramirrelevant}
Note that inequality \eqref{ugdefeq} is invariant under a change in parameter, and so it need only be verified on every curve that is parametrized by arc-length.
\end{remark}

\section{Analysis of upper gradients}
\label{uganalysis}

Our main result in this section, \cref{modgradthm}, characterizes the $L^p$-norms of minimal weak upper gradients in terms of the modulus of certain curve families.  Most of the other results here are found in surveys such as \cite{Hajlaszsurvey} and \cite{Heinonensurvey}, or have counterparts proved in \cite[Section 2]{Cheeger}, though \cref{r1ug,gradsandvariation} seem to be new, and are important for the proof of \cref{quasiconformal}.  

First, we note that being a weak gradient is a local condition.
\begin{lemma}
\label{localcondition}
A Borel function $g$ is a $p$-weak upper gradient for $\dfn$ if and only if for every $x\in X$, there is an open neighborhood $U$ of $x$ such that $\restrict[g]{U}$ is a $p$-weak upper gradient for $\restrict[\dfn]{U}$.
\end{lemma}
\begin{proof}
The first implication is trivial.  To prove the second, let $\{U_i\}$ be a countable basis for $X$ consisting of open neighborhoods on which $\restrict[g]{U_i}$ is a weak upper gradient for $\restrict[\dfn]{U_i}$.  Let $\Delta$ be the family of curves $\gamma \in\rectcurves(X)$ for which inequality \eqref{ugdefeq} fails, and for each $i$, define $\Delta_i\subset\Delta$ similarly, replacing $X$ with $U_i$.  

Suppose $\gamma\in\Delta$.  Then there are arbitrarily small subcurves of $\gamma$ in $\Delta$ as well, so by the compactness of $[a,b]$,  $\gamma$ has a subcurve in $\Delta_i$ for some $i$.  Thus $\Delta$ is minorized by $\bigcup_{i=1}^{\infty}\Delta_n$, whence by countable subadditivity and \cref{properties}, $\mod(\Delta)=0$.
\end{proof}

The usual definition of upper gradients is equivalent, via the following lemma, to an a priori stronger condition. The proof is a standard application of part \eqref{minorize} of \cref{properties} 
(see, e.g., the proof of \cite[Proposition 3.1]{Shanmugalingam}), and is thus omitted.
\begin{lemma}
\label{stronger}
A Borel function $g\colon X\rightarrow \BBR$ is a $p$-weak upper gradient for $\dfn$ if and only if for almost every every curve $\gamma \in\allcurves(X)$, inequality \eqref{ugdefeq} holds on every subcurve of $\gamma$.
\end{lemma}

We next characterize inequality \eqref{ugdefeq}  in terms of  $v_{\gamma}'$ and $v_{\dfn(\gamma)}'$.
\begin{lemma}
\label{r1ug}
Let $\gamma\in \abscurves(X)$, such that $\dfn(\gamma)\in \abscurves(Z)$.  If $g\colon X\rightarrow \BBR$ is a Borel function such that $\int_\gamma g\,ds<\infty$, then $g$ satisfies inequality \eqref{ugdefeq} for every subcurve of $\gamma$ if and only if the inequality
\begin{equation}
\label{r1ugdefeq}
g(\gamma(t))v_{\gamma}'(t) \geq v_{\dfn(\gamma)}'(t)
\end{equation}
holds for almost every $t\in [a,b]$.
\end{lemma}

\begin{proof}
Note that the upper gradient condition \eqref{ugdefeq} holds for a subcurve $\tilde{\gamma}=\restrict[\gamma]{[q,r]}$ of $\gamma$ if and only if $\int_{\tilde{\gamma}} g\,ds \geq v_{\dfn(\gamma)}(r) - v_{\dfn(\gamma)}(q)$.  Invoking equation \eqref{nonparam} and the fundamental theorem of calculus, this inequality becomes
\[ \int_q^r g(\gamma(t)) v_{\gamma}'(t)\, dt\geq \int_q^r v_{\dfn(\gamma)}'(t)\,dt\text{,} \]
which holds for every subinterval $[q,r]\subset[a,b]$ if and only if inequality \eqref{r1ugdefeq} is satisfied almost everywhere on [a,b].
\end{proof}

The following result was proved in \cite{Shanmugalingam}.
\begin{proposition}[{\cite[Proposition 3.1]{Shanmugalingam}}]
\label{ACC}
If $\dfn\in \newtonsobloc{1}{p}(X,Z)$, then $\dfn$ is absolutely continuous along almost every curve in $X$.
\end{proposition}
\begin{remark}
Although \cite[Proposition 3.1]{Shanmugalingam} assumes $Z=\BBR$, the proof carries over word for word to the general case.  Moreover, though the result there assumes a weak gradient in $L^p(X)$, the localization follows immediately from the countable subadditivity of $\mod$. 
\end{remark}

Combining \cref{ACC,stronger,r1ug,properties,paramirrelevant} yields the following characterization of weak upper gradients.

\begin{proposition}
\label{gradsandvariation}
Let $g\in L^p(X)$.  Then $g$ is a $p$-weak upper gradient for $\dfn$ if and only if for almost every curve $\gamma\in \abscurves(X)$, $\dfn$ is absolutely continuous along $\gamma$, and inequality \eqref{r1ugdefeq} holds almost everywhere.
\end{proposition}

It follows almost immediately from \cref{gradsandvariation} that weak upper gradients form a lattice, and behave well under restrictions. These properties are summarized in the next result (compare \cite[Lemma 7.17]{Hajlaszsurvey}).  
\begin{lemma}
\label{lattice}
Let $\dfn_1,\dfn_2\colon X\rightarrow Z$, for $p\geq 1$, let $A \subseteq X$ with $\restrict[\dfn_1]{A}=\restrict[\dfn_2]{A}$, and let $g_1,g_2\in \lloc{p}(X)$ be $p$-weak upper gradients for $\dfn_1$ and $\dfn_2$, respectively.  Then $g = \charfcn{X\backslash A} g_1 + \charfcn{A}\min(g_1,g_2)$ is a $p$-weak upper gradient for $\dfn_1$.  In particular, if $g_1$ and $g_2$ are $p$-weak upper gradients for $\dfn$, then so is $\min(g_1,g_2)$.
\end{lemma} 
\begin{proof}
By \cref{gradsandvariation}, for almost every curve $\gamma$, we have the inequality
\[ g_1(\gamma(t)) v_{\gamma}'(t)\geq  v_{\dfn_1(\gamma)}'(t)\]
for almost every $t\in \gamma^{-1}(X\backslash A)$, and, invoking \cref{curverestrict} as well,
\[ g_2(\gamma(t)) v_{\gamma}'(t)\geq v_{\dfn_2(\gamma)}'(t) = v_{\dfn_1(\gamma)}'(t)\]
for almost every $t\in \gamma^{-1}(A)$.   The lemma now follows, again from \cref{gradsandvariation}.
\end{proof}

The following two locality properties of $\grad{\dfn}$ are immediately deduced from \cref{lattice,localcondition}. Compare \cite[Corollary 2.25]{Cheeger}.
\begin{corollary}
\label{mingradrestrict}
If $\dfn_1,\dfn_2\in\newtonsobloc{1}{p}(X,Z)$, and
$\restrict[\dfn_1]{A}=\restrict[\dfn_2]{A}$ for some Borel set $A\subset X$, 
then $\restrict[\grad{\dfn_1}]{A}=\restrict[\grad{\dfn_2}]{A}$.
\end{corollary}

\begin{corollary}
\label{mingradlocal}
The map $\dfn$ is in the class $\newtonsobloc{1}{p}(X,Z)$ if and only if every $x\in X$ has an open neighborhood $U$ such that $\restrict[\dfn]{U} \in \newtonsobloc{1}{p}(U,Z)$. In this case, $\grad{(\restrict[\dfn]{U})}=\restrict[(\grad{\dfn})]{U}$ for every open subset $U\subseteq X$.
\end{corollary}

We now state the main result for this section. First, for every metric space $Z$, let $\bigcurves{\epsilon}(Z)$ be the collection of curves $\gamma\in\allcurves(Z)$ such that $|\gamma(b)-\gamma(a)|\geq \epsilon$.  Here and throughout, $\dfn^{-1}(\Gamma)$ denotes the family of curves $\gamma\in \allcurves(X)$ such that $\dfn(\gamma)\in \Gamma.$\footnote{Note that by this definition, a discontinuous map $\gamma\colon I\rightarrow X$ is never included in $\dfn^{-1}(\Gamma)$, even if $\dfn(\gamma)$ is a curve in $\Gamma$, which may occur if the map $\dfn$ is not a homeomorphism.}
\begin{theorem}
\label{modgradthm}
Let $\dfn\in L^p(X,Z)$, $p>1$.
Then $\dfn\in \newtonsob{1}{p}(X,Z)$ if and only if 
\begin{equation}
\label{finiteness}
\liminf_{\epsilon\rightarrow 0}\epsilon^p\mod(\dfn^{-1}(\bigcurves{\epsilon}(Z)))< \infty\text{.}
\end{equation}
Moreover, if this is the case, then the $\liminf$ on the left hand side is an actual limit, and
\begin{equation*}
\norm{\grad{\dfn}}^p = \lim_{\epsilon\rightarrow 0}\epsilon^p \mod(\dfn^{-1}(\bigcurves{\epsilon}(Z))) \text{.}
\end{equation*}
\end{theorem}

\begin{remark}
\cref{modgradthm} fails in the case $p=1$.  Indeed, let $\dfn \colon [0,1]\rightarrow [0,1]$ be a homeomorphism.  Then for each $n\in\BBN$, the function 
\begin{equation*}
\rho_n=\sum_{i=1}^{n} \frac{1}{\left|\dfn^{-1}\left(\frac{i}{n}\right)-\dfn^{-1}\left(\frac{i-1}{n}\right)\right|}\charfcn{\dfn^{-1}\left(\left[\frac{i-1}{n},\frac{i}{n}\right]\right)}
\end{equation*}
is admissible for $\dfn^{-1}(\bigcurves{\frac{2}{n}}([0,1]))$, and so
\begin{equation*}
\frac{2}{n}\mod(\dfn^{-1}(\bigcurves{\frac{2}{n}}([0,1])))\leq \frac{2}{n}\int_0^1 \rho_n(t)\,dt = 2,
\end{equation*}
and yet there are many homeomorphisms that are not in $\newtonsob{1}{1}([0,1])$.
\end{remark}

%

Before we prove \cref{modgradthm}, we address the issue of continuity (not absolute continuity) along almost every curve.  This is not necessary for our applications, where $\dfn=\map$ is a homeomorphism, but it may be of general interest that continuity need not be built into the hypotheses of \cref{modgradthm}.

Recall that $\dfn$ is said to be $\epsilon$-continuous if every point has a neighborhood $U$ such that $\diam(\dfn(U)) < \epsilon$.  We say that $\dfn$ is $\epsilon$-continuous along $\gamma$ if $\dfn(\gamma)$ is $\epsilon$-continuous.  
\begin{lemma}
\label{curvecontinuity}
Suppose $\mod(\dfn^{-1}(\bigcurves{\epsilon}(Z)))<\infty$.  Then $\dfn$ is $\epsilon$-continuous along $p$-almost every curve.  In particular, if $\mod(\dfn^{-1}(\bigcurves{\epsilon_n}(Z)))<\infty$ for some sequence $\{\epsilon_n\}$ converging to $0$, then $\dfn$ is continuous along $p$-almost every curve. 
\end{lemma}

\begin{proof}
Let $\rho\in L^p(X)$ be admissible for $\dfn^{-1}(\bigcurves{\epsilon}(Z))$.  Suppose $\dfn$ is not $\epsilon$-con\-tin\-u\-ous on a rectifiable curve $\gamma\colon [a,b]\rightarrow X$.  Note that $\epsilon$-continuity is preserved under a change of parameter, so we may assume that $\gamma=\arc{\gamma}$. 
Thus there is a point $t\in [a,b]$ and a sequence of points $t_i\in[a,b]$ (without loss of generality, with $t_i<t$) converging to $t$ such that $\dfn(\gamma(t_i))-\dfn(\gamma(t))\geq\epsilon$.  Thus for each $i$, $\restrict[\gamma]{[t_i,t]}\in \dfn^{-1}(\bigcurves{\epsilon}(Z))$, and so $\int_{t_i}^t \rho(\gamma(t))\, dt \geq 1$ for every $i$.  This implies $\int_\gamma \rho\,ds=\infty$, which, by \cref{properties} (part \eqref{intaecurve}), occurs only on an exceptional family of curves.

The final statement of the lemma follows from the countable subadditivity of modulus.
\end{proof}

\begin{proof}[Proof of \cref{modgradthm}]
Throughout this proof, we say that a function $\rho$ is \textit{almost admissible} for a curve family $\Gamma$ if $\rho$ is admissible for some subfamily $\tilde{\Gamma}\subset\Gamma$, with $\mod(\Gamma\backslash \tilde{\Gamma})=0$.  In this situation we have
\begin{equation*}
\mod(\Gamma) =  \mod(\tilde{\Gamma})\leq \int_X \rho^p\,d\meas \text{,}
\end{equation*}
so that from the point of view of estimating modulus, almost-admissible functions work as well as admissible ones.  Note also that by \cite[p. 182]{Fuglede}, there is a ``minimal'' almost-admissible function $\rho$, i.e., one such that $\mod(\Gamma)=\int_X \rho^p\,d\meas$.  Though the concept of almost-admissibility is not strictly necessary for our proof, it will simplify the exposition.

If $\dfn$ has a weak upper gradient in $L^p(X)$, and $\epsilon>0$, then $\epsilon^{-1}\grad{\dfn}$ is almost admissible for $\dfn^{-1}(\bigcurves{\epsilon}(Z))$, and so
\begin{equation*}
\limsup_{\epsilon\rightarrow 0}\epsilon^p \mod(\dfn^{-1}(\bigcurves{\epsilon}(Z))) \leq \norm{\grad{\dfn}}^p<\infty\text{.}
\end{equation*}
To complete the proof, we must show that if inequality \eqref{finiteness} is satisfied, then there is a weak upper gradient $g$ for $\dfn$ such that 
\begin{equation}
\label{modgradproof2}
\norm{g}^p \leq \liminf_{\epsilon\rightarrow 0}\epsilon^p \mod(\dfn^{-1}(\bigcurves{\epsilon}(Z))) \text{.}
\end{equation}

Let $\{\epsilon_n\}$ be a sequence, converging to $0$, such that 
\begin{equation*}
\label{finitenes1}
\lim_{n\rightarrow 0}\epsilon_n^p\mod(\dfn^{-1}(\bigcurves{\epsilon_n}(Z))) = \liminf_{\epsilon\rightarrow 0}\epsilon^p\mod(\dfn^{-1}(\bigcurves{\epsilon}(Z))) <\infty \text{.}
\end{equation*}
For each $n$, let $g_n$ be a Borel function such that $\epsilon_n^{-1}g_n$ is a minimal almost admissible function for $\dfn^{-1}(\bigcurves{\epsilon_n}(Z))$.  Note that this implies that for all $m\in \BBN$, $(m\epsilon_n)^{-1}g_n$ is almost admissible for $\dfn^{-1}(\bigcurves{m\epsilon_n}(Z))$.

We now construct $g$.  By inequality \eqref{finiteness}, the $L^p(X)$ norms $\norm{g_n}$ are bounded.  Thus, by the reflexivity of $L^p(X)$, Mazur's Lemma, and Fuglede's Theorem, there is a sequence of convex combinations $\omega_n= \sum_{i=1}^{l_n} \lambda^{i,n}g_{k_{i,n}}$ that converges (strongly) in $L^p(X)$ to $g$, such that for all $n$ and all $i\leq l_n$, we have $k_{i,n}\geq n$, and such that \begin{equation}
\label{aecurveconv}
\lim_{n\rightarrow\infty}\int_\gamma \omega_n\,ds= \int_\gamma g\,ds
\end{equation}
on almost every curve $\gamma$.  

We wish to show that $g$ is a weak upper gradient.  Let $\Gamma$ denote the family of rectifiable curves $\gamma$ with the following properties:  
\begin{itemize}
\item $\dfn$ is continuous on $\gamma$.
\item If $n,m\in \BBN$, and $\gamma_0$ is a subcurve of $\gamma$ such that $\gamma_0\in \dfn^{-1}(\bigcurves{m\epsilon_n}(Z))$, then
\[ \int_{\gamma_0} g_n \,ds \geq m\epsilon_n \text{.}\]
\item Equation \eqref{aecurveconv} is satisfied.
\end{itemize}

By part \eqref{minorize} of \cref{properties}, \cref{curvecontinuity}, and countable subadditivity, almost every curve is in $\Gamma$, and so it suffices to verify inequality \eqref{ugdefeq} for every $\gamma \in \Gamma$.

Fix a curve $\gamma\in \Gamma$.  For every $\epsilon>0$, let $t_\epsilon = \sup\{t\in[a,b]: |\dfn(\gamma(t))-\dfn(\gamma(a))|\in \BBN\epsilon\}$. By the continuity of $\dfn$ along $\gamma$, $|\dfn(\gamma(b))-\dfn(\gamma(t_\epsilon))|<\epsilon$, and also $|\dfn(\gamma(t_\epsilon))-\dfn(\gamma(a))|= m_\epsilon \epsilon$ for some $m_\epsilon\in \BBN$.  Thus by the second property, for every $n$ we have
\begin{align*}
&\int_\gamma g_n \,ds \geq \int_{\restrict[\gamma]{[a,t_{\epsilon_n}]}} g_n\,ds \geq m_{\epsilon_n} \epsilon_n \\
&=  |\dfn(\gamma(t_{\epsilon_n}))-\dfn(\gamma(a))| \geq |\dfn(\gamma(b))-\dfn(\gamma(a))|-\epsilon_n\text{.}
\end{align*}
In particular, we have the inequality
\[\liminf_{n\rightarrow\infty} \int_{\gamma} g_n\,ds \geq |\dfn(\gamma(b))-\dfn(\gamma(a))|\text{.}\]
It follows immediately that
\[\liminf_{n\rightarrow\infty} \int_{\gamma} \omega_n\,ds \geq |\dfn(\gamma(b))-\dfn(\gamma(a))|\text{,}\]
and so by the third property, 
$g$ satisfies inequality \eqref{ugdefeq} on $\gamma$. Thus $g$ is a weak upper gradient for $\dfn$, which satisfies inequality \eqref{modgradproof2} by construction.
\end{proof}

\section{Geometric vs. analytic quasiconformality.}
\label{quasiconformality}
\renewcommand{\expon}{Q}
In this section we prove \cref{quasiconformal,annularquasiconformal}.
\begin{proof}[Proof of \cref{quasiconformal}]
We prove the theorem by demonstrating that statements \eqref{adef} and \eqref{gdef} are equivalent to another condition.
\begin{enumerate}[(I)]
\setcounter{enumi}{\value{tempenumi}}
\item \label{mdef}
The inequality
\begin{equation*}
\mod(\map^{-1}(\bigcurves{\epsilon}(V)))\leq \epsilon^{-Q}K\mapmeas(V)
\end{equation*}
holds for every $\epsilon>0$, and every open subset $V\subset Y$. 
\setcounter{tempenumi}{\value{enumi}}
\end{enumerate}

\eqref{adef}$\Rightarrow$\eqref{gdef}.
Let $\Gamma\subset \allcurves(X)$, and let $\rho\colon Y\rightarrow [0,\infty]$ be admissible for $\map(\Gamma)$.  We may assume without loss of generality that $\rho\in L^Q(X)$, for otherwise, $\mod(\map(\Gamma))=\infty$ and there is nothing to prove.

We claim that $(\rho\circ \map)\grad{\map} $ is almost admissible for $\Gamma$.  Let $\gamma\in \Gamma$.  Invoking \cref{gradsandvariation}, and noting that the admissibility condition \eqref{admissibility} is independent of parametrization, we may assume that $\gamma$ is absolutely continuous, that $\map$ is absolutely continuous along $\gamma$, and that inequality \eqref{r1ugdefeq} holds almost everywhere on $[a,b]$.  We then have
\begin{align*}
&\int_{\gamma} (\rho\circ \map)\grad{\map} \,ds
= \int_a^b \rho(\map(\gamma(t)))\grad{\map}(\gamma(t)) v_\gamma'(t)\,dt\\
&\geq \int_a^b \rho(\map(\gamma(t)))v_{\map(\gamma)}'(t)\,dt
=\int_{\map(\gamma)}\rho\,ds\geq 1\text{.}
\end{align*}
Thus, as claimed, $(\rho\circ \map)\grad{\map}$ is almost admissible for $\Gamma$, and so
\begin{align*}
&\mod(\Gamma) 
\leq \int_X \rho(\map(x))^Q\grad{\map}(x)^Q\,d\meas(x)
\leq K \int_X \rho(\map(x))^Q \jac(x)\,d\meas(x)\\
&\leq K \int_X \rho(\map(x))^Q\,d\pullmeas(x) 
= \int_Y \rho^Q \,d\mapmeas\text{.}
\end{align*}
Since this holds for all admissible functions $\rho$ for $\map(\Gamma)$, we obtain
\begin{equation*}
\mod(\Gamma) \leq K\mod(\map(\Gamma))\text{.}
\end{equation*}

\eqref{gdef}$\Rightarrow$\eqref{mdef}.
For every $\epsilon>0$, $\epsilon^{-1}$ is an admissible function for $\bigcurves{\epsilon}(V)$, and therefore $\mod(\bigcurves{\epsilon}(V))\leq \epsilon^Q\mapmeas(V)$.  This, along with \eqref{gdef}, immediately yields \eqref{mdef}.

\eqref{mdef}$\Rightarrow$\eqref{adef}.
Whenever $U\subseteq X$ is an open set such that $\mapmeas(\map(U))$ is finite, we have
\begin{equation*}
\liminf_{\epsilon\rightarrow 0}\epsilon^Q\mod(\map^{-1}(\bigcurves{\epsilon}(\map(U))))\leq K \mapmeas(\map(U))<\infty\text{,}
\end{equation*}
so that by \cref{modgradthm}, $\restrict[\map]{U}\in \newtonsobloc{1}{Q}(U,Y)$, and
\begin{equation*}
\int_U \grad{\restrict[\map]{U}}^Q \,d\meas \leq K\mapmeas(\map(U))=K \pullmeas (U)\text{.}
\end{equation*}
By \cref{mingradlocal}, the separability of $Y$ and local finiteness of $\mapmeas$ then imply that $\map\in\newtonsobloc{1}{Q}(X,Y)$, and that
\begin{equation}
\label{adefeq1}
\int_U \grad{\map}^Q \,d\meas \leq K \pullmeas (U)\text{,}
\end{equation}
for every open subset $U\subseteq X$.  By the Borel regularity of $\pullmeas$, inequality \eqref{adefeq1} holds whenever $U\subset X$ is Borel, and so \eqref{adef} follows from the definitions of $\pullmeas$ and $\jac$.  
\end{proof}

\begin{remark}
\label{methodsremark}
In the proof of \eqref{mdef}$\Rightarrow$\eqref{adef}, our construction of an upper gradient, via \cref{modgradthm}, is similar to that of \cite[Section 4]{BaloghKoskelaRogovin}.  
In each case, a sequence of curve families allows approximation of the gradient, and reflexivity, Mazur's Lemma, and Fuglede's Theorem allow passage to a limit.
The key difference is in the choice of curve families.  The situation in \cite{BaloghKoskelaRogovin} requires the intersection of curves in $X$ with certain annuli, which necessitates the use of the families $\bigcurves{\epsilon}(X)$ to control of their diameters.  To construct an upper gradient via condition \eqref{gdef}, on the other hand, we must relate the admissibility condition \eqref{admissibility} with the upper gradient inequality \eqref{ugdefeq}, which in turn dictates the use of the families $\map^{-1}(\bigcurves{\epsilon}(Y))$ to control the diameter of curves in $Y$.  This need to control curve length in the target is a general consideration when working from geometric rather than metric assumptions, and is the reason we used annular condensers in $Y$ rather than $X$ in the statement of \cref{annularquasiconformal}. 
\end{remark}

\begin{remark}
\label{onesided}
Our results are fundamentally one-sided in the absence of a Poincar\'e inequality, even if the underlying spaces are $Q$-regular.  Indeed, if 
$\mod(\allcurves(X))=0$ (say, $Q=3$, $X=\BBR^2$ with the ``snowflaked'' metric $\dist(t_1,t_2)=|t_1-t_2|^{2/3}$), and $\mod(\allcurves(Y))>0$ (for example, ``Rickman's rug'' $Y=\BBR\times\BBR$ with the metric $\dist((s_1,t_1),(s_2,t_2))=|s_1-s_2|+|t_1-t_2|^{1/2}$), then for \textit{every} homeomorphism $\map\colon X\rightarrow Y$, $\map$ trivially satisfies the conditions of \cref{quasiconformal}, yet $\map^{-1}$ \textit{never} satisfies them.

For a less extreme example, let $X$ and $Y$ be the 
spaces from the previous paragraph.  Consider the $4$-regular spaces $X'=X\times \BBR$ and $Y'=Y\times \BBR$, where $\dist_{X'}((x_1,t_1),(x_2,t_2))= \dist_X(x_1,x_2)+|t_1-t_2|$, and similarly for $\dist_{Y'}$.  Equip each space with its Hausdorff $4$-measure. Let $\map\colon X'\rightarrow Y'$ be the identity (identifying each space as a set with $\BBR^3$).  The Hausdorff $4$-measure is invariant under isometries, and is thus a multiple of Lebesgue $3$-measure for each space.  Thus the Jacobian of $\map$ is some nonzero constant, say $\jac=C$.  

Unlike $X$, the space $X'$ has a somewhat healthy family of rectifiable curves. In fact, by Fubini's theorem, every subset of positive Hausdorff $4$-measure meets a curve family of positive $4$-modulus.  

The  map $\map$ is absolutely continuous on every $\gamma\in\rectcurves(X')$, and satisfies
$\grad{\map}\equiv 1$ (since $\rectcurves(X')$ consists only of curves of the form $\gamma(t)= (x,t)$ for some $x\in X$).  Hence $\map$ again satisfies \eqref{adef} and \eqref{gdef}, with $K=\frac{1}{C}$.  

On the other hand, the curve family  $\Gamma=\{\gamma_{(r,s)}\}\subset \allcurves(Y')$, where $\gamma_{(r,s)}(t) = (r, t, s)$ for $0\leq t\leq 1$, satisfies $\mod[4](\Gamma) >0$, again by Fubini.  These curves, however, have  unrectifiable pre-images, and so $\map^{-1}$ fails to be absolutely continuous on almost every curve, and thus does not satisfy condition \eqref{adef}.
\end{remark}

\begin{proof}[Proof of \cref{annularquasiconformal}.]
\eqref{gdef}$\Rightarrow$\eqref{cdef}.  The functions $\charfcn{\lambda B_{r}(y)}/(r(\lambda-1))$ are admissible for $\acon{y}{r}{\lambda r}$, giving the estimate $\mod(\acon{y}{r}{\lambda r})\leq Cr^{-Q}\mapmeas(B_r(y))$, where $C$ depends only on the doubling constant.  This, combined with condition \eqref{gdef}, immediately yields condition \eqref{cdef}, with $K'=CK$.

\eqref{cdef}$\Rightarrow$\eqref{mdef}. Fix $\epsilon>0$ and let $V\subset Y$ be open.  Since $Y$ is doubling, by a well-known covering lemma (see, e.g., \cite[Theorem 1.2]{Heinonen}) there is a countable family of balls $B_n=B_{r_n}(y_n)\subset V$ covering $V$, with $4 \lambda r_n<\epsilon$, such that the balls $B_n/5$ are disjoint, and such that $\mod(\map^{-1}(\aacon_n))\leq 2K'r_n^{-Q}\mapmeas(B_n)$, 
where $\aacon_n=\acon{y_n}{r_n}{\lambda r_n}$. Let $\rho_n$ be an admissible function for $\map^{-1}(\aacon_n)$ such that $\int_X \rho_n^Q \,d\meas\leq 2\mod(\map^{-1}(\aacon_n))$, and let $\rho(x)=4\epsilon^{-1}\lambda \sup_{n\in\BBN} r_n\rho_n(x)$.

We claim $\rho$ is admissible for $\map^{-1}(\bigcurves{\epsilon}(V))$.  Indeed, let $\gamma\in \map^{-1}(\bigcurves{\epsilon}(V))$.  Since the balls $B_n$ cover $V$, $\map(\gamma(a))$ lies in some ball $B_{n_1}$.  Since $\diam(\map(\gamma))>\epsilon/2 \geq 2 \lambda r_n\geq \diam(\lambda B_n)$, $\map(\gamma)\in \aacon_{n_1}$.  Let $a=t_0$, and let $t_1\in [a,b]$ be the first point in the interval such that $\map(\gamma_1)\in \aacon_{n_1}$, where $\gamma_1=\restrict[\gamma]{[t_0,t_1]}$.  Note that since $\map(\gamma(a))\in B_{n_1}$, our selection of $t_1$ means that $\map(\gamma_1)\subset \lambda B_{n_1}$, so that $\diam(\map(\gamma_1))\leq 2\lambda r_{n_1}$. If $\diam(\map(\restrict[\gamma]{[t_1,b]}))> \epsilon/2$, then we choose $n_2$, $t_2$ and $\gamma_2$ in the same manner as before, so that  $\map(\gamma_2)\in \aacon_{n_2}$, with $\gamma_2=\restrict[\gamma]{[t_1,t_2]}$.  We proceed this way until $\diam(\map(\restrict[\gamma]{[t_m,b]}))\leq\epsilon/2$.  Since $\gamma_i\in \map^{-1}(\aacon_{n_i})$, we estimate
\begin{align*}
&\int_\gamma \rho \,ds
\geq \sum_{i=1}^{m} \int_{\gamma_i} \rho\,ds
\geq 4\epsilon^{-1}\lambda \sum_{i=1}^{m} r_{n_i} \int_{\gamma_i} \rho_{n_i}\,ds
\geq 4\epsilon^{-1}\lambda \sum_{i=1}^{m} r_{n_i}\\
&\geq 2\epsilon^{-1} \sum_{i=1}^{m} \diam(\map(\gamma_i))
\geq 2\epsilon^{-1}(\epsilon/2)=1\text{,}
\end{align*}
and so $\rho$ is admissible for $\map^{-1}(\bigcurves{\epsilon}(V))$.  Therefore,
\begin{align*}
&\mod(\map^{-1}(\bigcurves{\epsilon}(V))
\leq \int_X \rho^Q \,d\meas
= (4\epsilon^{-1}\lambda)^Q \int_X \sup_{n\in\BBN} (r_n\rho_n)^Q \,d\meas\\
&\leq (4\epsilon^{-1}\lambda)^Q \int_X \sum_{n\in\BBN} (r_n\rho_n)^Q \,d\meas
\leq 2(4\epsilon^{-1}\lambda)^Q \sum_{n\in\BBN} r_n^Q \mod(\map^{-1}(\aacon_n))\\
&\leq 4K'(4\epsilon^{-1}\lambda)^Q \sum_{n\in\BBN} \mapmeas(B_n)
\leq CK' \epsilon^{-Q} \sum_{n\in\BBN} \mapmeas(B_n/5)
\leq CK' \epsilon^{-Q}\mapmeas(V)\text{,}
\end{align*}
where $C$ depends only on $\lambda$ and the doubling constant of $\mapmeas$. Setting $K=CK'$  completes the proof.
\end{proof}

\begin{remark}
\label{applicationsremark}
Suppose that at every $y\in Y$, condition \eqref{gdef} holds for the family $\Gamma=\map^{-1}(\acon{y}{r_i}{\lambda r_i})$, for some sequence of radii $r_i$ approaching $0$.  This implies condition \eqref{cdef}, and so to prove lower quasiconformality, it suffices to prove it for a sequence of inverse images of annular condensers at each point.

If $Y$ is locally linearly locally connected, then for some $\lambda_3> \lambda_2> \lambda_1>1$, $B_{r}(y)$ and $Y\backslash B_{\lambda_3r}(y)$ are contained, respectively, in connected components of $B_{\lambda_1 r}(y)$  and $Y\backslash B_{\lambda_2r}(y)$, provided $r$ is sufficiently small.  There is thus a ring (that is, a family of all the curves connecting two disjoint continua) $\mathcal R$ such that $\aacon=\acon{y}{r}{\lambda_3 r}$ is minorized by $\mathcal R$, which in turn is minorized by $\tilde{\aacon}=\acon{y}{\lambda_1 r}{\lambda_2 r}$. 

If condition \eqref{gdef} is satisfied for rings, we then have
\begin{align} 
\label{modineqs}
&\mod(\map^{-1}(\aacon))
\leq \mod(\map^{-1}(\mathcal R)) 
\leq K \mod(\mathcal R)\\
&\leq K \mod (\tilde{\aacon})
\leq KCr^{-Q}\mapmeas(B_r(y))\text{,}\notag
\end{align}
so that by \cref{annularquasiconformal}, $\map$ satisfies conditions \eqref{adef} and  \eqref{gdef} quantitatively, generalizing the classical result (specifically the ``$K_O(f)$'' part of \cite[Theorem 36.1]{Vaisala}). It is unclear whether the implication is sharp as in the classical case.

Finally, condition \eqref{cdef} can be verified without difficulty when $X$ and $Y$ are Ahlfors $Q$-regular and $\map$ is quasisymmetric.  Indeed, $\map^{-1}(\acon{y}{r}{\lambda_1 r})$ is minorized by $\acon{\map^{-1}(y)}{r'}{\lambda_2 r'}$ for some radius $r'$, where $\lambda_1$ and $\lambda_2$ depend only on the function $\eta$ in the definition of quasisymmetry, recalled in \cref{lipminimal} below. 
Thus
\begin{align*} 
&\mod(\map^{-1}(\acon{y}{r}{\lambda_1 r}))
\leq \mod(\acon{\map^{-1}(y)}{r'}{\lambda_2 r'})\\
&\leq C_1(r')^{-Q}\meas(B_{r'}(\map^{-1}(y)))
\leq C_2\leq C_3r^{-Q}\mapmeas(B_r(y))\text{,}
\end{align*}
giving a short proof of Tyson's theorem \cite[Theorem 1.4]{Tyson} on the geometric quasiconformality of quasisymmetric maps.  This is not quite a new proof, however;  the methods in \cite{BaloghKoskelaRogovin} give a similar construction of the upper gradient, as discussed in \cref{methodsremark}.
\end{remark}

\section{Pointwise outer dilatation and P. I. spaces}
\label{lipminimal}
In this section we prove \cref{sharpening}.
For an arbitrary map $\dfn\colon X\rightarrow Z$, the \textit{approximate pointwise Lipschitz constant} of $\dfn$ is
\begin{equation*}
\aplip{\dfn}(x) = \inf_A \limsup_{x'\rightarrow x, x'\in A\backslash\{x\}} \frac{ |\dfn(x')- \dfn(x)|}{|x'-x|} \text{,}
\end{equation*}
with the infimum  taken over subsets $A\subseteq X$ having a Lebesgue point of density at $x$.  In general, if $x$ is isolated, we let $\lip{\dfn}=\aplip{\dfn}(x)=0$.
By \cite[Proposition 3.5]{KeithPI}, $\aplip{\dfn}=\lip{\dfn}$ for a locally Lipschitz function $\dfn$, provided $\meas$ is doubling.  One can check (see, e.g., \cite[Example 3.15]{HKST}) that $\lip{\dfn}$ is an upper gradient for a locally Lipschitz function $\dfn$.  

We say $X$ is a \textit{$p$-P.I. space} if it is complete, doubling, and admits a weak $(1,p)$-Poincar\'e inequality.  We say $Z$ has \textit{curvature bounded above} if it has that property (in the sense of Alexandrov), is locally compact, and is locally geodesically complete. Definitions and background on each type of space can be found, respectively, in \cite{HeinonenKoskela} and \cite{BBI}.

When $X$ is a P.I. space, and $\dfn$ is locally Lipschitz, $\lip{\dfn}$ and $\grad{\dfn}$ are at least comparable, regardless of $Z$; by \cite[Proposition 4.26]{Cheeger}, at almost every $x\in X$, 
\begin{equation}
\label{lipcomparable}
\frac{1}{C}\lip{\dfn}(x)\leq\grad{\dfn}(x)\leq\lip{\dfn}(x)\text{.}
\end{equation}
Here $C$ depends only on the constants associated with the doubling condition and Poincar\'e inequality for $X$.  \footnote{In \cite{Cheeger}, $Z=\BBR$ is assumed throughout, but the assumption is not  used for this result.  Occurrences of \cite[(4.3)]{Cheeger} need only be replaced with the Poincar\'e inequality for Banach space valued maps, via \cite[Theorem 4.3]{HKST}.}  

Without the assumption that $\dfn$ is locally Lipschitz, inequality \eqref{lipcomparable} need not hold \cite[Remark 2.16]{KeithPI}.  However, if $\dfn\in \newtonsobloc{1}{p}(X,Z)$, then $\dfn$ is locally Lipschitz off of sets of arbitrary small measure (see, e.g., \cite[Lemma 10.7]{HKST} and the preceding remarks).  Via the Kuratowski embedding \cite[Exercise 12.5]{Heinonen}, we may assume that $Z=l^{\infty}$, and so we may use the McShane extension \cite[Theorem 6.2]{Heinonen} and \cref{mingradrestrict} to conclude the following 
generalization of \cite[Proposition 4.26]{Cheeger}.
\begin{lemma}
\label{aplipcomparable}
Let $X$ be a $p$-P.I. space, $p>1$, and let $\dfn\in\newtonsobloc{1}{p}(X,Z)$, for any metric space $Z$.  Then for $\meas$-almost every $x\in X$,
\begin{equation*}
\frac{1}{C}\aplip{\dfn}(x)\leq\grad{\dfn}(x)\leq\aplip{\dfn}(x)\text{,}
\end{equation*} 
where $C$ is a constant depending only on $p$ and the data of $X$.
\end{lemma}

If $Z=\BBR$, or more generally, has curvature bounded above, a theorem of Cheeger, along with generalizations due to Keith and Ohta, equates $\grad{\dfn}$ to $\lip{\dfn}$, or to $\aplip{\dfn}$, whenever $\dfn$ is a locally Lipschitz or Sobolev map, respectively.   
\begin{theorem}[{\cite[Theorem 6.1]{Cheeger}, \cite[Remark 2.16]{KeithPI}, \cite[Theorem 5.9]{Ohta}}]
\label{cheegerlip}
Let $p>1$, and let $\dfn\in \newtonsobloc{1}{p}(X,Z)$, where $X$ is a $p$-P.I. space and $Z$ has curvature bounded above.  Then $\grad{\dfn}=\aplip{\dfn}$.  In particular, if $\dfn$ is locally Lipschitz, then $\grad{\dfn}=\lip{\dfn}$.
\end{theorem}

\begin{proof}
For the case $Z=\BBR$, the theorem was proved in \cite{Cheeger} for locally Lipschitz functions, and later in \cite{KeithPI} for Newton-Sobolev functions.  The extension to targets with curvature bounded above was proved, in the Lipschitz case, in \cite{Ohta}.  The Newton-Sobolev case for these targets requires only a small modification of the argument there, which we now give. 

Let $d_z(z')=|z-z'|$, for $z,z'\in Z$.  Let $Z_0\subset Z$ be countable and dense, and note that $g$ is a weak upper gradient for $\dfn$ if and only if it is a weak upper gradient for $d_z\circ\dfn$ for each $z\in Z_0$. Since the theorem holds for real valued functions, and since, by \cref{aplipcomparable}, $\aplip{\dfn}(x)<\infty$ almost everywhere, it suffices to show that
$\aplip{\dfn}(x)=\sup_{z\in Z_0} \aplip{(d_z\circ \dfn)}(x)$ for every $x\in X$ such that $\aplip{\dfn}(x)<\infty$.  

Fix such a point $x$, fix $\epsilon>0$, and let $U$ be a normal neighborhood of $\dfn(x)$.  
By the compactness of the set of directions from $x$, and the density of $Z_0$, there is a finite subset $S_\epsilon\subset Z_0\cap U$ such that for all $z\in U\backslash\{\dfn(x)\}$, there exists $s\in S_{\epsilon}$ such that $\cos(\angle z\dfn(x)s) \geq 1-\epsilon$. By the definition of $\aplip{\dfn}$ and finiteness of $S_\epsilon$, there is a subset $A\subset X$ such that $x$ is a Lebesgue point of $A$, and for each $s\in S_\epsilon$, 
\begin{equation*}
\aplip{(d_s\circ\dfn)}(x) \geq \limsup_{x'\rightarrow x, x'\in A\backslash\{x\}} \frac{ |d_s(\dfn(x'))- d_s(\dfn(x))|}{|x'-x|} -\epsilon \text{.}
\end{equation*}
The argument from \cite[Lemma 5.4]{Ohta} then shows that for some $s\in S_\epsilon$,
\[\aplip{(d_s\circ \dfn)}(x)\geq (1-\epsilon)\aplip{\dfn}(x)-\epsilon\text{.}\]
Letting $\epsilon\rightarrow 0$ completes the proof.  Note that where the argument in \cite{Ohta} invokes the local Lipschitz property, it suffices to use the fact that by our choice of $x$, we may assume that the ratio $|\dfn(x')-\dfn(x)|/|x'-x|$ is bounded on $A$.
\end{proof}


\begin{remark}
We do not know whether \cref{cheegerlip} holds for arbitrary metric space targets $Z$.  The argument in \cite{Cheeger} relies on the reflexivity of the Cheeger-Sobolev spaces $\cheegersob{1}{p}(X,\BBR)$, and so does not seem to generalize to other targets.  Our discussion in this section shows that the final statement in \cref{sharpening} is valid whenever the conclusion of \cref{cheegerlip} holds for $Z=Y$.
\end{remark}

Next, suppose $\map\colon X\rightarrow Y$ is quasisymmetric.  Recall that this means there is a homeomorphism $\eta\colon [0,\infty)\rightarrow [0,\infty)$ such that 
\[
\frac{|\map(x_3)-\map(x_2)|}{|\map(x_3)-\map(x_1)|} \leq \eta \left(\frac{|x_3-x_2|}{|x_3-x_1|}\right)
\]
for every three distinct points $x_1,x_2,x_3\in X$.  
\begin{lemma}
\label{qslip}
Suppose $\meas$ is doubling, and let $\map\colon X\rightarrow Y$ be a quasisymmetric embedding.  Then for every $x\in X$, $\aplip{\map}(x)=\lip{\map}(x)$.
\end{lemma}
\begin{proof}
The statement is trivial if $x$ is isolated.  Suppose, then, that $x$ is not isolated, let $\{x_n\}$ be a sequence of points in $X$ with $\lim x_n = x$, and let $x$ be a Lebesgue point of a Borel set $A\subset X$.  After passing to a subsequence, we may assume there are points $x_n'\in A$ such that $|x_n'-x_n|\leq \min(|x_n-x|/n,|x_n'-x|/n)$. We then have
\begin{align*}
&\frac{|\map(x_n)-\map(x)|}{|x_n-x|}
= \left(\frac{|\map(x_n)-\map(x)|}{|\map(x_n')-\map(x)|}\right)
\left(\frac{|\map(x_n')-\map(x)|}{|x_n'-x|}\right)
\left(\frac{|x_n'-x|}{|x_n-x|}\right)\\
&\leq \left(\eta\left(\frac{1}{n}\right)+1\right)
\left(\frac{|\map(x_n')-\map(x)|}{|x_n'-x|}\right)
\left(\frac{1}{n}+1\right)\text{,}
\end{align*}
and so the lemma follows upon passing to the limit supremum as $n\rightarrow \infty$.
\end{proof}
Combining \cref{cheegerlip,aplipcomparable,qslip}, we obtain the following result.
\begin{proposition}
\label{piqslip}
Let $X$ be a $p$-P.I. space, $p>1$. If $\map\in\newtonsobloc{1}{p}(X,Y)$ is a quasisymmetric homeomorphism, then for $\meas$-almost every $x\in X$,
\begin{equation*}
\frac{1}{C}\lip{\dfn}(x)\leq\grad{\dfn}(x)\leq\lip{\dfn}(x)\text{,}
\end{equation*} 
where the constant $C$ depends only on $p$ and the data of $X$, and not on the space $Y$.  If $Y$ has curvature bounded above, then $\grad{\dfn}(x)=\lip{\dfn}(x)$  $\meas$-almost everywhere.
\end{proposition}
\begin{proof}[Proof of \cref{sharpening}.]
Let $\map$ be as in \cref{sharpening}. By \cite[Theorem 5.12]{HeinonenKoskela},  $X$ is a $Q$-P.I. space.  Since $\meas$ is doubling, $X$ is a Vitali space \cite[Remark 1.13]{Heinonen}, so
\begin{equation*}
\jac(x)=\lim_{r\rightarrow 0} \frac{\mapmeas(\map(B_r(x)))}{\meas(B_r(x))}\text{,}
\end{equation*}
almost everywhere \cite[2.9.2]{Federer}, and so $H_O(x,\map)=\lip{\map}(x)^Q/\jac(x)$ almost everywhere.  By condition \eqref{gdef} and \cite[Theorem 9.8]{HKST}, $\map$ is locally quasisymmetric.    
The theorem now follows from \cref{piqslip,quasiconformal}.

\end{proof}

\bibliographystyle{amsplain}
\bibliography{refrepository}

\end{document}